\documentclass[a4paper,11pt]{amsart}
\pdfoutput=1

\usepackage[utf8]{inputenc}
\usepackage{mathtools,amssymb}
\usepackage{hyperref}
\usepackage{xcolor,graphicx}
\usepackage{mathrsfs}
\usepackage{enumitem}

\graphicspath{{images/}}

\newtheorem{theorem}{Theorem}[section]
\newtheorem{lemma}[theorem]{Lemma}

\newtheorem{corollary}[theorem]{Corollary}
\newtheorem{remark}[theorem]{Remark}

\title[X-ray tomography of one-forms with partial data]{X-ray tomography of one-forms with partial data}

\author{Joonas Ilmavirta}

\thanks{Unit of Computing Sciences, Tampere University, Kalevantie 4, FI-33014 Tampere University, Finland; \href{mailto:joonas.ilmavirta@tuni.fi}{joonas.ilmavirta@tuni.fi}}
\author{Keijo M\"onkk\"onen}
\thanks{Department of Mathematics and Statistics, University of Jyv\"askyl\"a, P.O. Box 35 (MaD) FI-40014 University of Jyv\"askyl\"a, Finland; \href{mailto:kematamo@student.jyu.fi}{kematamo@student.jyu.fi}}

\date{\today}

\usepackage{graphicx}
\newcommand{\cev}[1]{\reflectbox{\ensuremath{\vec{\reflectbox{\ensuremath{#1}}}}}}

\newcommand{\R}{{\mathbb R}}

\newcommand{\N}{{\mathbb N}}

\newcommand{\der}{{\mathrm d}}
\newcommand*{\sol}[1]{#1^\mathrm{s}}

\newcommand{\xrt}{X}%X-ray transform
\newcommand{\no}{N}%normal operator

\newcommand{\schwartz}{\mathscr{S}}
\newcommand{\tempered}{\mathscr{S}^{\prime}}

\newcommand{\csmooth}{\mathcal{D}}%compactly supported smooth function
\newcommand{\smooth}{\mathcal{E}}%smooth function
\newcommand{\cdistr}{\mathcal{E}'}%compactly supported distribution
\newcommand{\distr}{\mathcal{D}^{\prime}}%distribution
\newcommand{\dimens}{n}%dimension
\newcommand{\abs}[1]{\left\lvert #1 \right\rvert}%absolute value
\newcommand{\ip}[2]{\left\langle #1,#2 \right\rangle}%inner product or duality pairing
\DeclareMathOperator{\spt}{spt}%support
\DeclareMathOperator{\conv}{Conv}%convex hull
\DeclareMathOperator{\curl}{curl}
\DeclareMathOperator{\diver}{div}

\newcommand{\NTR}[1]{}

\begin{document}

\maketitle

\begin{abstract}
If the integrals of a one-form over all lines meeting a small open set vanish and the form is closed in this set, then the one-form is exact in the whole Euclidean space. We obtain a unique continuation result for the normal operator of the X-ray transform of one-forms, and this leads to one of our two proofs of the partial data result. Our proofs apply to compactly supported covector-valued distributions.
\end{abstract}

\NTR{Updates to the manuscript have been indicated with these footnotes. We are most grateful to the referees for their feedback and suggestions, and we hope the present version addresses all their concerns.}
\NTR{Affiliation has been updated.}

\section{Introduction}

Let~$f$ be a one-form on $\R^\dimens$ where $\dimens\geq 2$.
We define the X-ray transform (also known as the Doppler transform in this case) of~$f$ by the formula
\begin{equation}
\xrt_1 f(\gamma)=\int_{\gamma}f
\end{equation}
where~$\gamma$ is a line in~$\R^\dimens$.
We freely identify one-forms with vector fields, so the differential of a scalar field corresponds to its gradient.
We are interested in the problem of reconstructing~$f$ from~$\xrt_1 f$.
One-forms of the form $f=\der\phi$ where~$\phi$ goes to zero at infinity are always in the kernel of~$\xrt_1$.
Thus one can only try to recover the solenoidal part~$\sol{f}$ of the \NTR{We removed the word ``solenoidal" in ``solenoidal decomposition".} decomposition $f=\sol{f}+\der\phi$ from the data~$\xrt_1 f$.
The transform~$\xrt_1$ is known to be solenoidally injective~\cite{PSU-tensor-tomography-progress, SHA-integral-geometry-tensor-fields}, i.e. $\xrt_1 f=0$ implies $f=\der\phi$ for some scalar function~$\phi$.
We study whether this implication holds in the whole space also in the case where we know~$\xrt_1 f$ only for a subset of lines.

We consider the following partial data problem for~$\xrt_1$. Let $V\subset\R^\dimens$ be a nonempty open set.
Assume that we know~$\der f|_V$ and~$\xrt_1 f$ on all lines intersecting~$V$, where~$\der f$ is the exterior derivative or the curl of the one-form~$f$.
Can we determine the solenoidal part~$\sol{f}$ -- find~$f$ modulo potential fields -- from this data?
We will study the uniqueness of the partial data problem: If $\der f|_V=0$ and $\xrt_1 f=0$ on all lines intersecting~$V$, does it follow that~$\sol{f}=0$?

The partial data problem for~$\xrt_1$ can be reduced to the following unique continuation problem for the normal operator $\no_1=\xrt_1^*\xrt_1$:
if $\der f|_V=0$ and $\no_1 f|_V=0$, does it follow that $\sol{f}=0$?
We prove that such unique continuation property holds for compactly supported covector-valued distributions under the weaker assumption that~$\no_1 f$ vanishes to infinite order at some point in~$V$.
The unique continuation of the normal operator implies uniqueness for the partial data problem:
The solenoidal part of a one-form~$f$ is uniquely determined whenever one knows the curl of the one-form in~$V$ and the integrals of~$f$ over all lines intersecting~$V$.

For scalar fields the uniqueness of a corresponding partial data problem and the unique continuation of the normal operator were proved in~\cite{IM-unique-continuation-riesz-potential}.
We generalize the results to one-forms using the results for scalar fields in our proofs.
We also obtain partial data results and unique continuation results for the generalized X-ray transform of one-forms $\xrt_A=\xrt_1\circ A$ where~$A$ is a smooth invertible matrix-valued function.
As a special case of this transform we study the transverse ray transform in~$\R^2$.

We give two alternative proofs for the partial data results.
The first one uses the unique continuation of the normal operator while the second one works directly at the level of the X-ray transform and is based on Stokes' theorem. 

The X-ray transform of one-forms or vector fields has applications in the determination of velocity fields of moving fluids using acoustic travel time measurements~\cite{NO-tomographic-recostruction-of-vector-fields} or Doppler backscattering measurements~\cite{NO-unique-tomographic-reconstruction-doppler}.
Medical applications include ultrasound imaging of blood flows~\cite{JASESPL-ultrasound-doppler-tomography, JUH-principles-of-doppler-tomography, SSLP-doppler-tomography-vector-fields}.
The transverse ray transform of one-forms has applications in the temperature measurements of flames~\cite{BH-tomographic-reconstruction-vector-fields, SCHWA-flame-analysis-schlieren}.
For two-tensors the applications include also diffraction tomography~\cite{LW-diffraction-tomography}, photoelasticity~\cite{HL-applications-to-photoelasticity} and polarization tomography~\cite{SHA-integral-geometry-tensor-fields}.
For a more comprehensive treatment see the reviews~\cite{SCHU-20-years-of-vector-tomography, SCHU-importance-of-vector-field-tomography, SS-vector-field-overview} and the references therein.

We will give our main results in section~\ref{sec:mainresults} and discuss related results in section~\ref{sec:relatedresults}.
The preliminaries are covered in section~\ref{sec:preliminaries} and finally the theorems are proven in section~\ref{sec:proofsofthemainresults}.

\subsection{Main results}
\label{sec:mainresults}
Here we give the main results of this paper.
The proofs can be found in section~\ref{sec:proofsofthemainresults}.
First we briefly go through our notation; for more detailed definitions see section~\ref{sec:preliminaries}.

\NTR{Added ``...see \textbf{equation} \eqref{eq:normaloperatorvectorfield}...".}Let $\cdistr(\R^\dimens)$ be the space of compactly supported distributions.
By $f\in(\cdistr(\R^\dimens))^\dimens$ we mean that $f=(f_1, \dotso, f_\dimens)$ where $f_i\in\cdistr(\R^\dimens)$ for all $i=1, \dotso, \dimens$.
We call $(\cdistr(\R^\dimens))^\dimens$ the space of compactly supported covector-valued distributions.
We denote by~$\xrt_1$ the X-ray transform of one-forms and by $\no_1=\xrt_1^*\xrt_1$ its normal operator; see equation~\eqref{eq:normaloperatorvectorfield} for an explicit formula.

We say that~$\no_1f$ vanishes to infinite order at~$x_0\in\R^\dimens$ if it is smooth in a neighborhood of~$x_0$ and $\partial^{\beta}(\no_1f)_i(x_0)=0$ for all $\beta\in\N^\dimens$ and $i=1, \dotso, \dimens$.
We denote the exterior derivative of differential forms by~$\der$.
When acting on scalars, it corresponds to the gradient.

%We denote by~$\der f$ the exterior derivative whose components are $(\der f)_{ij}=\partial_i f_j-\partial_j f_i$.

\NTR{Added ``...for scalar fields and the \textbf{normal} operator..."}Our first result is a unique continuation property for the normal operator~$\no_1$.
The corresponding result for scalar fields and the normal operator $\no_0=\xrt_0^*\xrt_0$ of the scalar X-ray transform~$\xrt_0$ (see equation~\eqref{eq:normaloperatorofscalarfield}) was proven in~\cite[Theorem 1.1]{IM-unique-continuation-riesz-potential}.

\begin{theorem}
\label{thm:uniquecontinuationofnormaloperator}
Let $f\in (\cdistr(\R^\dimens))^\dimens$ and $V\subset\R^\dimens$ some nonempty open set.
If $\der f|_V=0$ and~$\no_1 f$ vanishes to infinite order at $x_0\in V$, then $f=\der\phi$ for some $\phi\in\cdistr(\R^\dimens)$.
%$\sol{f}=0$.
\end{theorem}

We point out that as~$\der f$ vanishes in~$V$, the distribution~$\no_1f$ is smooth in~$V$ by lemma~\ref{lma:N1smooth} and the vanishing condition at a point is well-defined.

Theorem~\ref{thm:uniquecontinuationofnormaloperator} is also true under the weaker assumption that $\der f|_V=0$ and $\der (\no_1 f)$ vanishes to infinite order at~$x_0$ (see the proof in section~\ref{subsec:proofsusingucp}).
The condition that~$f$ is closed in~$V$ (i.e. $\der f|_V=0$) is satisfied if, for example, $f|_V=0$.
When~$f$ is solenoidal (i.e. $\diver(f)=0$), theorem~\ref{thm:uniquecontinuationofnormaloperator} gives the following unique continuation property: if $f|_V=\no_1 f|_V=0$, then $f=0$.

The next result is stated directly at the level of the X-ray transform.
The corresponding problem with full data was solved in~\cite[Theorem 2.5.1]{SHA-integral-geometry-tensor-fields}.

\begin{theorem}
\label{thm:globalpartialdataresult}
Let $f\in (\cdistr(\R^\dimens))^\dimens$ and $V\subset\R^\dimens$ some nonempty open set. Assume that $\der f|_V=0$.
Then $\xrt_1 f$ vanishes on all lines intersecting~$V$ if and only if $f=\der\phi$ for some $\phi\in\cdistr (\R^\dimens)$.% and $\spt(\phi)$ is contained in the convex hull of $\spt(f)$. 
\end{theorem}

\begin{remark}
\label{rmk:spt}
In theorems~\ref{thm:uniquecontinuationofnormaloperator} and~\ref{thm:globalpartialdataresult} the support of the potential~$\phi$ is contained in the convex hull of~$\spt(f)$.
\end{remark}

\begin{remark}
\NTR{We added this remark which shows how one can prove partial data results for certain functions on the sphere bundle S$\R^n$ using partial data results for scalar fields and one-forms.}We can combine the partial data result for vector fields (theorem~\ref{thm:globalpartialdataresult}) with the partial data result for scalar fields (lemma~\ref{lemma:partialdataproblemscalar}) to obtain the following partial data result. Let $F\colon S\R^\dimens\rightarrow\R$ be a function on the sphere bundle~$S\R^\dimens=\R^\dimens\times S^{\dimens-1}$ defined as $F(x, \xi)=g(x)+f(x)\cdot\xi$ where $g\colon\R^\dimens\rightarrow\R$ is a function on~$\R^\dimens$ and~$f$ is a one-form on~$\R^\dimens$. We define the X-ray transform $\xrt_{S\R^\dimens}$ of~$F$ as
\begin{equation}
\xrt_{S\R^\dimens}F(\gamma)=\int_{\R}F(\gamma(t), \dot{\gamma}(t))\der t=\xrt_0 g(\gamma)+\xrt_1 f (\gamma)
\end{equation}
where~$\gamma$ is an oriented line in~$\R^n$ and~$\xrt_0$ is the X-ray transform of scalar fields (see section~\ref{sec:xraytransformscalar}). 

Assume that $V\subset\R^\dimens$ is a nonempty open set such that $g|_V=\der f|_V=0$ and $\xrt_{S\R^\dimens}F(\gamma)=0$ on all lines~$\gamma$ intersecting~$V$. Denote by $\overleftarrow{\gamma}$ the reversed line. Since $\xrt_0 g(\overleftarrow{\gamma})=\xrt_0 g(\gamma)$ and $\xrt_1 f(\overleftarrow{\gamma})=-\xrt_1 f(\gamma)$ we obtain $\xrt_0 g(\gamma)=\frac{1}{2}(\xrt_{S\R^\dimens}F(\gamma)+\xrt_{S\R^\dimens}F(\overleftarrow{\gamma}))$ and $\xrt_1 f(\gamma)=\frac{1}{2}(\xrt_{S\R^\dimens}F(\gamma)-\xrt_{S\R^\dimens}F(\overleftarrow{\gamma}))$. Hence the partial data problem for~$\xrt_{S\R^\dimens}F$ decouples to separate partial data problems for~$\xrt_0 g$ and~$\xrt_1 f$. Using theorem~\ref{thm:globalpartialdataresult} and lemma~\ref{lemma:partialdataproblemscalar} one obtains that $g=0$ and $f=\der\phi$ for some scalar field~$\phi$. This means that $F=\der\phi$, i.e. $F(x, \xi)=\der\phi(x)\cdot\xi$. See~\cite{AD-finsler-scalar-plus-oneform, SU-attenuated-scalar-plus-one-form} for similar results in the case of full data.
\end{remark}

One can view theorems~\ref{thm:uniquecontinuationofnormaloperator} and~\ref{thm:globalpartialdataresult} in terms of the global solenoidal decomposition $f=\sol{f}+\der\phi$ (see section~\ref{subsec:distributionsandsolenoidaldecomposition} and equation~\eqref{eq:globalsolenoidaldecomposition}).
The conclusion $f=\der\phi$ for some $\phi\in\cdistr(\R^\dimens)$ is equivalent to $\sol{f}=0$.

From theorem~\ref{thm:globalpartialdataresult} we obtain the following local partial data result in a bounded domain $\Omega\subset\R^\dimens$.
The X-ray transform of $f\in (L^2(\Omega))^\dimens$ is defined to be $\xrt_1 f:=\xrt_1\widetilde{f}$ where~$\widetilde{f}$ is the zero extension of~$f$ to~$\R^\dimens$.

\begin{theorem}
\label{thm:localpartialdataresult}
Let $f\in (L^2(\Omega)))^\dimens$ where $\Omega\subset\R^\dimens$ is a bounded and smooth convex domain and let $V\subset\Omega$ be some nonempty open set.
Assume that $\der f|_V=0$.
Then $\xrt_1 f=0$ on all lines intersecting~$V$ if and only if $f=\der\phi$ for some $\phi\in H^1_0(\Omega)$.
\end{theorem}

In terms of the local solenoidal decomposition $f=\sol{f}_{\Omega}+\der\phi_\Omega$ (see section~\ref{subsec:distributionsandsolenoidaldecomposition} and equation~\eqref{eq:poissonequation}) the conclusion $f=\der\phi$ for some $\phi\in H^1_0(\Omega)$ is equivalent to that $\sol{f}_\Omega=0$. 

From theorem~\ref{thm:uniquecontinuationofnormaloperator} we also obtain the following unique continuation and partial data results for the transform~$\xrt_A=\xrt_1\circ A$ where~$A=A(x)$ is smooth invertible matrix field.
We denote by~$\no_A=A^T\circ\no_1\circ A$ the normal operator of~$\xrt_A$. When~$B$ is the constant matrix field $B(v_1, v_2)=(v_2, -v_1)$ where $(v_1, v_2)\in\R^2$ we write $\xrt_B=\xrt_\perp$ and call~$\xrt_\perp$ the transverse ray transform.

\begin{corollary}
\label{cor:ucpofgeneralnormaloperator}
Let $f\in (\cdistr(\R^\dimens))^\dimens$ and $V\subset\R^\dimens$ some nonempty open set. If $\der (Af)|_V=0$ and $\no_A f|_V=0$, then $f=A^{-1}(\der\psi)$ for some $\psi\in\cdistr (\R^\dimens)$.
\end{corollary}
 
\begin{corollary}
\label{cor:partialdatageneraltransform}
Let $f\in (\cdistr(\R^\dimens))^\dimens$ and $V\subset\R^\dimens$ some nonempty open set. Assume that $\der (Af)|_V=0$.
Then~$\xrt_A f$ vanishes on all lines intersecting~$V$ if and only if $f=A^{-1}(\der\psi)$ for some $\psi\in\cdistr (\R^\dimens)$.
\end{corollary}

\NTR{Added references to the articles~\cite{AJM-v-line-transform} and~\cite{DS-tomography}.}In corollaries~\ref{cor:ucpofgeneralnormaloperator} and~\ref{cor:partialdatageneraltransform} the distribution~$\psi\in\cdistr(\R^\dimens)$ is the potential part of the solenoidal decomposition of~$Af\in (\cdistr(\R^\dimens))^\dimens$ and~$\spt(\psi)$ is contained in the convex hull of~$\spt(f)$.
As a special case of the transform~$\xrt_A$ we obtain the next partial data result for the transverse ray transform~$\xrt_\perp$ which is similar to the full data result in~\cite{BH-tomographic-reconstruction-vector-fields, DS-tomography} (see also~\cite{AJM-v-line-transform}).

\begin{corollary}
\label{cor:transverseraytransform}
Let $f\in (\cdistr(\R^2))^2$ and $V\subset\R^2$ some nonempty open set.
Assume that $\diver (f)|_V=0$.
Then~$\xrt_\perp f$ vanishes on all lines intersecting~$V$ if and only if $\diver(f)=0$.

In particular, if $\der f|_V=\diver(f)|_V=0$ and both~$\xrt_1f$ and~$\xrt_\perp f$ vanish on all lines intersecting~$V$, then $f=0$.
\end{corollary}

Alternatively, one can conclude in the first claim of corollary~\ref{cor:transverseraytransform} that $f=\curl(\psi)$ for some $\psi\in\cdistr(\R^\dimens)$ where $\curl(\psi)=(\partial_2\psi, -\partial_1\psi)$. In terms of the global solenoidal decomposition this is equivalent to that $f=\sol{f}$. Also in the latter claim it is enough to know the partial data of~$\xrt_1 f$ for~$V\subset\R^2$ and the partial data of~$\xrt_\perp f$ for
~$W\subset\R^2$ where~$V$ and~$W$ can be disjoint.  

\begin{remark}
Some of the results above can be slightly generalized.
Using the same proof as in theorem~\ref{thm:localpartialdataresult} one can show that corollaries~\ref{cor:partialdatageneraltransform} and~\ref{cor:transverseraytransform} hold also in the local case when $f\in (L^2(B))^\dimens$.
Also in corollary~\ref{cor:ucpofgeneralnormaloperator} one can replace the condition $\no_Af|_V=0$ with the requirement that~$\no_Af$ vanishes to infinite order at $x_0\in V$ when~$A$ is a constant matrix field.
Especially, this holds for the normal operator of the transverse ray transform.
One can also see from theorem~\ref{thm:globalpartialdataresult} and corollary~\ref{cor:transverseraytransform} that the X-ray transform and the transverse ray transform provide complementary information about the one-form in~$\R^2$.
\end{remark}

We note that if~$A=A(x)$ is not invertible for all $x\in\R^\dimens$, we can still conclude in corollary~\ref{cor:partialdatageneraltransform} that $Af=\der\psi$ for some potential $\psi\in\cdistr(\R^\dimens)$.
Thus we obtain the ``pointwise projection"~$Af$ modulo potentials from the local data for~$\xrt_A f$.
We also remark that in all of our results which consider the X-ray transform in~$\R^\dimens$ we could replace the assumption of compact support with rapid decay at infinity.
If all the derivatives of the matrix field $A=A_{ij}(x)$ grow at most polynomially, then the results are true for one-forms whose\NTR{Changed `which' to `whose'.} components are Schwartz functions.
This follows since the corresponding partial data result for scalar fields holds for Schwartz functions~\cite{IM-unique-continuation-riesz-potential} and our method of proof is based on reducing the problem of one-forms to the problem of scalar fields.

\subsection{Related results}
\label{sec:relatedresults}
Similar partial data results as in theorems~\ref{thm:globalpartialdataresult} and~\ref{thm:localpartialdataresult} are previously known for scalar fields. If one knows the values of the scalar function~$f$ in an open set~$V$, then one can uniquely reconstruct~$f$ from its local X-ray data \cite{CNDK-solving-interior-problem-ct-with-apriori-knowledge, IM-unique-continuation-riesz-potential, KEQ-wavelet-methods-ROI-tomography}. In~$\R^2$ uniqueness is also obtained under weaker assumptions: if~$f$ is piecewise constant, piecewise polynomial or analytic in~$V$, then one can recover~$f$ uniquely from its integrals over the lines going through~$V$ \cite{KKW-stability-of-interior-problems, KEQ-wavelet-methods-ROI-tomography, YYJW-high-order-TV-minimization}. A complementary partial data result is the Helgason support theorem~\cite{HE:integral-geometry-radon-transforms}. According to Helgason's theorem, if $f|_C=0$ and the integrals of~$f$ vanish on all lines not intersecting a compact and convex set~$C$, then $f=0$. 

The normal operator of the X-ray transform of scalar fields admits a similar unique continuation property as in theorem~\ref{thm:uniquecontinuationofnormaloperator}. If~$f$ is a function which satisfies $f|_V=0$ and
~$\no_0 f$ vanishes to infinite order at some point in~$V$, then $f=0$~\cite{IM-unique-continuation-riesz-potential}. This is a special case of a more general unique continuation result for Riesz potentials~\cite{IM-unique-continuation-riesz-potential} (see equation~\eqref{eq:normaloperatorofscalarfield}). Unique continuation of Riesz potentials is related to unique continuation of fractional Laplacians~\cite{CMR-ucp-higher-order-laplacians, GSU-calderon-problem-fractional-schrodinger, IM-unique-continuation-riesz-potential} (see also equation~\eqref{eq:invertingnormaloperatorscalar}).

\NTR{We explained the support theorem for~$\xrt_1$ in more detail. We corrected the additional assumption $f|_C=0$ to $\der f|_C=0$.}
Unique reconstruction of the solenoidal part of a one-form or vector field with full data is known in~$\R^\dimens$ \cite{JUH-principles-of-doppler-tomography, NO-tomographic-recostruction-of-vector-fields, SSLP-doppler-tomography-vector-fields, SU:microlocal-analysis-integral-geometry} and on compact simple Riemannian manifolds with boundary~\cite{IM:integral-geometry-review, PSU-tensor-tomography-progress}.
In~$\R^\dimens$ uniqueness holds for compactly supported covector-valued distributions as well
~\cite{SHA-integral-geometry-tensor-fields}.
Some partial data results are known for one-forms.
The solenoidal part can be reconstructed by knowing
~$\xrt_1 f$ on all lines parallel to a finite set of planes~\cite{JUH-principles-of-doppler-tomography, SCHU-3d-doppler-transform-reconstruction-and-kernels, SHA-vector-tomography-incomplete-data}.
When $n\geq 3$, one can locally recover one-forms up to potential fields near a strictly convex boundary point~\cite{SUV-local-invertibility-on-tensors}, and the solenoidal part can be determined from the knowledge of~$\xrt_1 f$ on all lines intersecting a certain type of curve~\cite{VER-integral-geometry-symmetric-tensor-incomplete} (see also~\cite{DEN-inversion-of-3d-tensor-fields}).
\NTR{We replaced ``...the solenoidal part can be locally recovered near a strictly convex boundary point..." with ``one can locally recover one-forms up to potential fields near a strictly convex boundary point". We continued this sentence as ``...and the solenoidal part can be determined..." instead of ``...or...".}
One can also obtain information about the singularities of the curl of a compactly supported covector-valued distribution from its X-ray data on lines intersecting a fixed curve~\cite{RA-microlocal-analysis-doppler-transform}. There is a Helgason-type support theorem for the X-ray transform of one-forms which is in a sense complementary to our result. If~$f$ integrates to zero over all lines not intersecting a compact and convex set~$C$, then $\der f=0$ in the complement of~$C$ \cite[Theorem 7.5]{SU:microlocal-analysis-integral-geometry}. If we further assume that $\der f|_C=0$, then the one-form~$f$ is closed in the whole space which implies that~$f$ is exact and the solenoidal part of~$f$ vanishes. See also the discussion after the alternative proof in section~\ref{sec:alternativeproof}. 

\NTR{Added references to the articles~\cite{AJM-v-line-transform} and~\cite{DS-tomography}.}The transverse ray transform has been studied earlier with full data in~$\R^2$~\cite{BH-tomographic-reconstruction-vector-fields, DS-tomography, NA-mathematical-methods-image-reconstruction} and also on Riemannian manifolds~\cite{IMR-mixing-ray-transforms, SHA-integral-geometry-tensor-fields} (see also~\cite{ABH-support-theorem-transverse-ray} for a support theorem). The transverse ray transform is a special case of a more general mixed ray transform~\cite{dESUZ-generic-uniqueness-mixed-ray, deSZ-mixed-ray, DS-tomography, SHA-integral-geometry-tensor-fields}.
In higher dimensions the transverse ray transform is related to the normal Radon transform~\cite{SS-vector-field-overview, STR-reconstruction-from-doppler-radon-tranforms}.
In
~$\R^2$ and on certain Riemannian manifolds the knowledge of~$\xrt_1 f$ and~$\xrt_\perp f$ fully determines the one-form~\cite{BH-tomographic-reconstruction-vector-fields, DS-tomography, IMR-mixing-ray-transforms}.
By theorem~\ref{thm:globalpartialdataresult} and corollary~\ref{cor:transverseraytransform} this is true in~$\R^2$ also in the case of partial data.
In higher dimensions~$f$ is determined by~$\xrt_1 f$ and the normal Radon transform of~$f$
~\cite{STR-reconstruction-from-doppler-radon-tranforms}.
A similar transform to~$\xrt_A$ was studied in~\cite{IMR-mixing-ray-transforms, PRI-inner-product-probes}. Recently in~\cite{AJM-v-line-transform} the authors studied the so-called V-line transform of vector fields which is a generalization of the X-ray transform to V-shaped ``lines" which consist of one vertex and two rays (half-lines).

\subsection*{Acknowledgements}
J.I. was supported by the Academy of Finland (grants 332890 and 336254).\NTR{Added new grant.}
K.M. was supported by Academy of Finland (Centre of Excellence in Inverse Modelling and Imaging, grant numbers 284715 and 309963).
We are grateful to Lauri Oksanen for discussions.
The authors want to thank the anonymous referees for their valuable comments.
\NTR{We added our thanks to the referees. Thank you!}

\section{Preliminaries}
\label{sec:preliminaries}
In this section we give a brief introduction to the theory of X-ray tomography of scalar fields and one-forms in~$\R^\dimens$. We also define the generalized X-ray transform of one-forms.
First we recall the definition and solenoidal decomposition of covector-valued distributions. 
We mainly follow the conventions of~\cite{deRham-differentiable-manifolds, HO-topological-vector-spaces, NA-mathematics-computerized-tomography, SHA-integral-geometry-tensor-fields, SU:microlocal-analysis-integral-geometry, TRE:topological-vector-spaces-distributions} and refer the reader to them for further details.

\subsection{Covector-valued distributions and solenoidal decomposition}
\label{subsec:distributionsandsolenoidaldecomposition}
We denote by~$\csmooth(\R^\dimens)$ the space of compactly supported smooth functions, by~$\schwartz(\R^\dimens)$ the space of rapidly decreasing smooth functions (Schwartz space) and by~$\smooth (\R^\dimens)$ the space of smooth functions. All spaces are equipped with their standard topologies. The spaces~$\distr (\R^\dimens)$, $\tempered(\R^\dimens)$ and~$\cdistr(\R^\dimens)$ are the corresponding topological duals. Elements of~$\distr(\R^\dimens)$ are called distributions and~$\cdistr(\R^\dimens)$ can be seen as the space of compactly supported distributions. We have the continuous inclusions $\cdistr(\R^\dimens)\subset\tempered(\R^\dimens)\subset\distr(\R^\dimens)$.
We write the dual pairing as $\ip{f}{\varphi}$ when~$f$ is a distribution and~$\varphi$ is a test function.
%$f\in\distr (\R^\dimens)$ and $\varphi$ is a test function.

We define the vector-valued test function space~$(\csmooth(\R^\dimens))^\dimens$ such that $\varphi\in (\csmooth(\R^\dimens))^\dimens$ if and only if $\varphi=(\varphi_1, \dotso, \varphi_\dimens)$ and $\varphi_i\in\csmooth(\R^\dimens)$ for all $i=1, \dotso, \dimens$. The topology of the space~$(\csmooth(\R^\dimens))^\dimens$ is defined as follows: a sequence~$\varphi_k$ converges to zero in~$(\csmooth(\R^\dimens))^\dimens$ if and only if
~$(\varphi_k)_i$ converges to zero in
~$\csmooth(\R^\dimens)$ for all $i=1, \dotso, \dimens$. We then define the space of covector-valued distributions~$(\distr(\R^\dimens))^\dimens$ so that $f\in(\distr(\R^\dimens))^\dimens$ if and only if $f=(f_1, \dotso, f_n)$ and $f_i\in\distr (\R^\dimens)$ for all $i=1, \dotso , \dimens$. The duality pairing of $f\in (\distr(\R^\dimens))^\dimens$ and $\varphi\in (\csmooth(\R^\dimens))^\dimens$ becomes
\begin{equation}
\ip{f}{\varphi}=\sum_{i=1}^\dimens\ip{f_i}{\varphi_i}.
\end{equation}
The spaces~$(\smooth(\R^\dimens))^\dimens$, $(\schwartz(\R^\dimens))^\dimens$, $(\cdistr(\R^\dimens))^\dimens$ and~$(\tempered(\R^\dimens))^\dimens$ are defined in a similar way and we call~$(\cdistr(\R^\dimens))^\dimens$ the space of compactly supported covector-valued distributions. Covector-valued distributions are a special case of currents which are continuous linear functionals in the space of differential forms~\cite[Section III]{deRham-differentiable-manifolds}. The components of the exterior derivative or the curl of a one-form or covector-valued distribution are
\begin{equation}
(\der f)_{ij}=\partial_i f_j-\partial_j f_i.
\end{equation}

One can split certain covector-valued distributions into a divergence-free part and a potential part. If $f\in(\cdistr(\R^\dimens))^\dimens$, then we have the unique decomposition~\cite{SHA-integral-geometry-tensor-fields}
\begin{equation}
\label{eq:globalsolenoidaldecomposition}
f=\sol{f}+\der\phi, \quad \diver(\sol{f})=0
\end{equation}
where $\phi\in\tempered(\R^\dimens)$ and $\sol{f}\in(\tempered(\R^\dimens))^\dimens$ are smooth outside~$\spt(f)$ and go to zero at infinity. Here~$\phi$ is defined so that it solves the equation $\Delta\phi=\diver(f)$ in the sense of distributions and $\sol{f}=f-\der\phi$. The decomposition~\eqref{eq:globalsolenoidaldecomposition} is known as solenoidal decomposition or Helmholtz decomposition and it holds also for $f\in (\schwartz(\R^\dimens))^\dimens$~\cite{SHA-integral-geometry-tensor-fields}. We call~$f$ solenoidal if $\diver(f)=0$. For the decomposition~\eqref{eq:globalsolenoidaldecomposition} this means that $f=\sol{f}$.

If~$f$ is supported in a fixed set, we can do the decomposition locally in that set. If $\Omega\subset\R^\dimens$ is a regular enough bounded domain and $f\in (L^2(\Omega))^\dimens$, we let~$\phi_{\Omega}$ to be the unique weak solution to the Poisson equation 
\begin{equation}
\label{eq:poissonequation}
\begin{cases}
\Delta\phi=\diver(f) \ \text{in} \ \Omega \\
\phi\in H^1_0(\Omega).
\end{cases}
\end{equation}
Then we have $f=\sol{f}_{\Omega}+\der\phi_{\Omega}$ where $\sol{f}_{\Omega}=f-\der\phi_{\Omega}\in (L^2(\Omega))^\dimens$ and $\diver(\sol{f}_{\Omega})=0$. If $f\in (C^{1, \alpha}(\overline{\Omega}))^\dimens$ for some $0<\alpha<1$, then there is unique classical solution $\phi_{\Omega}\in C^{2, \alpha}(\overline{\Omega})$ to the boundary value problem~\eqref{eq:poissonequation} and the solenoidal decomposition holds pointwise~\cite{GT-elliptic-pdes}.

\subsection{The X-ray transform of scalar fields}
\label{sec:xraytransformscalar}

Let~$\Gamma$ be the set of all oriented lines in~$\R^\dimens$.
The X-ray transform of a function~$f$ is defined as
\begin{equation}
\xrt_0 f(\gamma)=\int_{\gamma}f\der s, \quad \gamma\in\Gamma
\end{equation}
whenever the integrals exist.
The set~$\Gamma$ can be parameterized as
\begin{equation}
\label{eq:parametrizationoflines}
\Gamma=\{(z, \theta): \theta\in S^{\dimens-1}, \ z\in\theta^{\perp}\}.
\end{equation}
Then the X-ray transform becomes
\begin{equation}
\xrt_0 f(z, \theta)=\int_{\R}f(z+t\theta)\der t
\end{equation}
and it is a continuous map $\xrt_0\colon\csmooth(\R^\dimens)\rightarrow\csmooth(\Gamma)$. One can define the adjoint using the formula
\begin{equation}
\xrt_0^*\psi(x)=\int_{S^{\dimens-1}}\psi(x-(x\cdot\theta)\theta, \theta)\der\theta
\end{equation}
and it follows that $\xrt_0^*\colon\smooth(\Gamma)\rightarrow\smooth(\R^\dimens)$ is continuous. By duality we can define $\xrt_0\colon\cdistr(\R^\dimens)\rightarrow\cdistr(\Gamma)$ and $\xrt_0^*\colon\distr(\Gamma)\rightarrow\distr(\R^\dimens)$ as
\begin{align}
\langle \xrt_0 f, \varphi\rangle&=\langle f, \xrt_0^*\varphi\rangle \\
\langle \xrt_0^*g, \eta\rangle&=\langle g, \xrt_0\eta\rangle.
\end{align}

\NTR{We wrote an integral formula for $\no_0f(x)$ so that the notation for convolution can be understood. We also replaced ``By duality the formula \eqref{eq:normaloperatorofscalarfield}..." with ``By duality the formula $\no_0f=2(f\ast\abs{\cdot}^{1-n})$...".}
The normal operator $\no_0=\xrt_0^*\xrt_0$ is useful in studying the properties of the X-ray transform since it takes functions on~$\R^\dimens$ to functions on~$\R^\dimens$. It has an expression
\begin{equation}
\label{eq:normaloperatorofscalarfield}
\no_0 f(x)
=
2\int_{\R^\dimens}\frac{f(y)}{\abs{x-y}^{\dimens-1}}\der y
=
2(f\ast\abs{\cdot}^{1-\dimens})(x)
\end{equation}
for continuous functions~$f$ decreasing rapidly enough at infinity. By duality the formula $\no_0f=2(f\ast\abs{\cdot}^{1-n})$ holds also for compactly supported distributions and the normal operator becomes a map $\no_0\colon\cdistr(\R^\dimens)\rightarrow\distr(\R^\dimens)$. One can invert~$f$ from its X-ray transform using the normal operator by
\begin{equation}
\label{eq:invertingnormaloperatorscalar}
f=c_{0, \dimens}(-\Delta)^{1/2}\no_0f,
\end{equation}
where $c_{0, \dimens}=(2\pi\abs{S^{n-2}})^{-1}$ is a constant depending on the dimension and $(-\Delta)^{1/2}$ is the fractional Laplacian of order~$1/2$.
The inversion formula~\eqref{eq:invertingnormaloperatorscalar} holds for $f\in\cdistr(\R^\dimens)$ and for continuous functions~$f$ decreasing rapidly enough at infinity.

\subsection{The X-ray transform of one-forms}

\NTR{We explained what the formula~\eqref{eq:xraytransformofvectorfield} means. We also replaced ``...parametrization~\eqref{eq:parametrizationoflines} for $\Gamma$ we obtain..." with ``...parametrization~\eqref{eq:parametrizationoflines} for~$\Gamma$ we \textbf{can write}...".}
Let~$f$ be a one-form on~$\R^\dimens$.
We define its X-ray transform as
\begin{equation}
\label{eq:xraytransformofvectorfield}
\xrt_1 f(\gamma)=\int_{\gamma}f, \quad \gamma\in\Gamma
\end{equation}
whenever the integrals exist. The formula~\eqref{eq:xraytransformofvectorfield} is understood as the integral of the one-form~$f$ over the (oriented) one-dimensional submanifold~$\gamma$. Using the parametrization~\eqref{eq:parametrizationoflines} for~$\Gamma$ we can write
\begin{equation}
\xrt_1 f(z, \theta)=\int_{\R}f(z+t\theta)\cdot\theta\der t.
\end{equation}
It follows that $\xrt_1\colon(\csmooth(\R^\dimens))^\dimens\rightarrow\csmooth(\Gamma)$ is continuous. The adjoint is defined as
\begin{equation}
(\xrt_1^*\psi)_i(x)=\int_{S^{\dimens-1}}\theta_i \psi(x-(x\cdot\theta)\theta, \theta)\der\theta
\end{equation}
and $\xrt_1^*\colon\smooth(\Gamma)\rightarrow(\smooth(\R^\dimens))^\dimens$ is also continuous. Thus we can define $\xrt_1\colon(\cdistr(\R^\dimens))^\dimens\rightarrow\cdistr(\Gamma)$ and $\xrt_1^*\colon\distr(\Gamma)\rightarrow(\distr(\R^\dimens))^\dimens$ as
\begin{align}
\ip{\xrt_1 f}{\varphi}&=\ip{f}{\xrt_1^*\varphi} \\
\ip{\xrt_1^* g}{\eta}&=\ip{g}{\xrt_1\eta}.
\end{align}
If $f\in (L^p(\Omega))^\dimens$ where $\Omega\subset\R^\dimens$ is a bounded domain and $p\geq 1$, we define its X-ray transform as $\xrt_1 f:=\xrt_1\widetilde{f}$ where $\widetilde{f}\in (\cdistr(\R^\dimens))^\dimens$ is the zero extension of~$f$.

Like in the scalar case we define the normal operator $\no_1=\xrt_1^*\xrt_1$ and it satisfies the formula
\begin{equation}
\label{eq:normaloperatorvectorfield}
(\no_1 f)_i=\sum_{j=1}^\dimens\frac{2x_ix_j}{\abs{x}^{\dimens+1}}\ast f_j.
\end{equation}
The normal operator can be extended to a map $\no_1\colon (\cdistr(\R^\dimens))^\dimens\rightarrow (\distr(\R^\dimens))^\dimens$ and the formula~\eqref{eq:normaloperatorvectorfield} holds for $f\in (\cdistr(\R^\dimens))^\dimens$ and also for continuous one-forms decreasing rapidly enough at infinity. One can invert the solenoidal part of~$f$ using the normal operator by
\begin{equation}
\label{eq:invertingsolenoidalpartfromnormaloperator}
\sol{f}=c_{1, \dimens}(-\Delta)^{1/2}\no_1 f,
\end{equation}
where $c_{1, \dimens}=\abs{S^n}$ is a constant depending on the dimension and $(-\Delta)^{1/2}$ operates componentwise. The formula~\eqref{eq:invertingsolenoidalpartfromnormaloperator} holds for $f\in(\cdistr(\R^\dimens))^\dimens$ and also for continuous one-forms decreasing rapidly enough at infinity.

\subsection{The generalized X-ray transform of one-forms}
Let $A=A(x)$ be a smooth matrix-valued function on~$\R^\dimens$ such that for each $x\in\R^\dimens$ the matrix~$A(x)$ is invertible. We define the transform~$\xrt_A$ of a one-form~$f$ as
\begin{equation}
\xrt_Af(\gamma)
=
\int_{-\infty}^{\infty}A(\gamma(t))f(\gamma(t))\cdot\dot{\gamma}(t)\der t=\xrt_1(A f)(\gamma), \quad \gamma\in\Gamma.
\end{equation}
Thus~$\xrt_A$ can be seen as the X-ray transform of the ``rotated" one-form~$Af$. The transform~$\xrt_A$ can also be defined on compactly supported covector-valued distributions. We first let $\ip{A f}{\varphi}=\ip{f}{A^T\varphi}$ for $f\in (\distr(\R^\dimens))^\dimens$ and a test function~$\varphi$ where~$A^T$ is the pointwise transpose of~$A$ and $(A^T\varphi)(x)=A^T(x)\varphi(x)$. Then clearly~$A$ is a map $A\colon (\cdistr(\R^\dimens))^\dimens\rightarrow (\cdistr(\R^\dimens))^\dimens$. Therefore we can define~$\xrt_A\colon (\cdistr(\R^\dimens))^\dimens\rightarrow \distr(\Gamma)$ as $\xrt_A f=\xrt_1(Af)$. One easily sees that the adjoint is $\xrt^*_A=A
^T\circ\xrt^*_1$ and the normal operator becomes $\no_A=A^T\circ\no_1\circ A$. By the discussion above the normal operator can be extended to a map $\no_A\colon (\cdistr(\R^\dimens))^\dimens\rightarrow (\distr(\R^\dimens))^\dimens$.

\NTR{Added reference to the article~\cite{DS-tomography}.}Let~$B$ be the constant matrix field on~$\R^2$ defined as $B(v_1e_1+v_2e_2)=v_2e_1-v_1e_2$ where $\{e_1, e_2\}$ is any orthonormal basis of~$\R^2$. The matrix~$B$ corresponds to a clockwise rotation by 90 degrees. We then define the transverse ray transform~$\xrt_\perp$ by letting $\xrt_\perp=\xrt_B$. It follows that the transverse ray transform provides complementary information about the solenoidal decomposition compared to the X-ray transform, i.e.~$\xrt_1$ determines the solenoidal part and~$\xrt_\perp$ determines the potential part of a one-form~\cite{BH-tomographic-reconstruction-vector-fields, DS-tomography} (see also theorem~\ref{thm:globalpartialdataresult} and corollary~\ref{cor:transverseraytransform}). 

\section{Proofs of the main results}
\label{sec:proofsofthemainresults}

We give two alternative proofs for the partial data results. The first proof uses the unique continuation of the normal operator and the second proof works directly at the level of the X-ray transform. Both proofs are based on the corresponding results for scalar fields.

\subsection{Proofs using the unique continuation of the normal operator}
\label{subsec:proofsusingucp}

In this section we prove our main results using the unique continuation property of the normal operator.
We need the following lemmas in our proofs.

\begin{lemma}[{\cite[Theorem 1.1]{IM-unique-continuation-riesz-potential}}]
\label{lemma:ucpofscalarnormaloperaor}
Let $V\subset\R^\dimens$ be some nonempty open set and $g\in\cdistr(\R^\dimens)$. If $g|_V=0$ and $\partial^{\beta}(\no_0 g)(x_0)=0$ for some $x_0\in V$ and all $\beta\in\N^\dimens$, then $g=0$.
\end{lemma}

\begin{lemma}[Poincar\'e lemma]
\label{lemma:poincarelemma}
Let $g\in(\distr (\R^\dimens))^
\dimens$ such that $\der g=0$.
Then there is $\eta\in\distr (\R^\dimens)$ such that $\der\eta=g$.
If $g\in(\cdistr (\R^\dimens))^\dimens$, then $\eta\in\cdistr (\R^\dimens)$.
%If~$g$ is smooth, then so is~$\eta$.
\end{lemma}

The proof of lemma~\ref{lemma:poincarelemma} can be found in~\cite{HO-topological-vector-spaces, MA-poincare-derham-theorems}.
We first prove the unique continuation result for the normal operator.
The proof is based on the fact that we can reduce the unique continuation problem of~$\no_1$ to a unique continuation problem of~$\no_0$ acting on the components of~$\der f$.

The assumptions of theorem~\ref{thm:uniquecontinuationofnormaloperator} come in two stages.
We first assume that $\der f|_V=0$.
To make sense of the next assumption that~$\no_1f$ vanishes at~$x_0$ to infinite order, we need to ensure that it is smooth near this point.
This is given by the next lemma.

\begin{lemma}
\label{lma:N1smooth}
Let $V\subset\R^\dimens$ be an open set and $f\in(\cdistr(\R^\dimens))^\dimens$.
If $\der f|_V=0$, then $\no_1f|_V$ is smooth.
\end{lemma}

\begin{proof}
Take any $x_0\in V$ and a small open ball~$B$ centered at it and contained in~$V$.
As $\der f|_B=0$, the Poincar\'e lemma applied in the ball~$B$ (lemma~\ref{lemma:poincarelemma} is applicable because~$B$ is diffeomorphic to~$\R^n$) gives $f|_B=\der h$ for some $h\in\distr(B)$.
Let $B'\subset B$ be a smaller ball with the same center, and let $\chi\in\csmooth(B)$ be a bump function so that $\chi|_{B'}\equiv1$.
If we let $h'=\chi h\in\cdistr(\R^\dimens)$, then $f=\der h'+g$, where $g\in(\cdistr(\R^\dimens))^\dimens$ with $g|_{B'}=0$.

As $\xrt_1(\der h')=0$ (cf.~\eqref{eq:xraytransformofpotential}), we have $\no_1f=\no_1g$.
Because $g|_{B'}=0$, it follows from properties of convolutions that~$\no_1f$ is smooth in~$B'$.
Now that~$\no_1f$ is smooth in a neighborhood of any point in~$V$, the claim follows.
\end{proof}

\begin{proof}[Proof of theorem~\ref{thm:uniquecontinuationofnormaloperator}]
The normal operator has an expression 
\begin{equation}
(\no_1 f)_i
=
\sum_{j=1}^\dimens\frac{2x_ix_j}{\abs{x}^{\dimens+1}}\ast f_j.
\end{equation}
We can write the kernel as
\begin{equation}
\frac{2x_ix_j}{\abs{x}^{n+1}}
=
\frac{2}{\dimens-1}\bigg(\delta_{ij}\abs{x}^{1-\dimens}-\partial_i(x_j\abs{x}^{1-\dimens})\bigg)
%\frac{2}{1-\dimens}\bigg(\partial_i(x_j\abs{x}^{1-\dimens})-\delta_{ij}\abs{x}^{1-\dimens}\bigg)
\end{equation}
and we obtain
\begin{equation}
(\no_1 f)_i
=
%\frac{2}{1-\dimens}\bigg(\sum_{j=1}^\dimens (x_j\abs{x}^{1-\dimens}\ast\partial_i f_j)-\frac{1}{2}\no_0 f_i\bigg)
\frac{2}{\dimens-1}\bigg(
\frac{1}{2}\no_0 f_i
-
\sum_{j=1}^\dimens x_j\abs{x}^{1-\dimens}\ast\partial_i f_j
\bigg)
.
\end{equation}
We can calculate that
\begin{equation}
\label{eq:relationofnormaloperators}
\partial_k (\no_1f)_i-\partial_i(\no_1 f)_k
=
\frac{1}{\dimens-1}
\no_0(\partial_k f_i-\partial_i f_k)
.
\end{equation}
This can be interpreted as $
\der(\no_1f)
=
(\dimens-1)^{-1}
\no_0(\der f)
$,
where the scalar normal operator~$\no_0$ acts on the $2$-form~$\der f$ componentwise to produce another $2$-form.
The normal operator commutes with the exterior derivative in this sense.

Since~$\no_1 f$ vanishes to infinite order at $x_0\in V$ also $\no_0(\partial_k f_i-\partial_i f_k)$ vanishes to infinite order at~$x_0$.
Using lemma~\ref{lemma:ucpofscalarnormaloperaor} we obtain $\der f=0$.
By lemma~\ref{lemma:poincarelemma} there is $\phi\in\cdistr (\R^\dimens)$ such that $\der \phi=f$.
This concludes the proof.
\end{proof}

Lemma~\ref{lemma:ucpofscalarnormaloperaor} is false if no restrictions are imposed on~$g|_V$ \cite{KEQ-wavelet-methods-ROI-tomography, NA-mathematics-computerized-tomography}, and the assumption $g|_V=0$ is the most convenient.
Consequently, the assumption $\der f|_V=0$ in theorem~\ref{thm:uniquecontinuationofnormaloperator} is important.
This condition is invariant under gauge transformations of the field~$f$.

If $\sol{f}|_V=\no_1 f|_V=0$, then one can alternatively use the unique continuation of the fractional Laplacian $(-\Delta)^s$, $s\in (0, 1)$, to prove the unique continuation of the normal operator~\cite{GSU-calderon-problem-fractional-schrodinger}.
This follows since $(-\Delta)^{1/2}\sol{f}=c_{1, \dimens}(-\Delta)\no_1 f$ where $\sol{f}\in (H^r (\R^\dimens))^\dimens$ for some $r\in\R$ when $f\in(\cdistr(\R^\dimens))^
\dimens$.
One can also make use of the fact that $(-\Delta)^{-1/2}$ is a Riesz potential and use its unique continuation properties~\cite{IM-unique-continuation-riesz-potential} (see equation~\eqref{eq:invertingsolenoidalpartfromnormaloperator}).

The rest of the results follow easily from theorem~\ref{thm:uniquecontinuationofnormaloperator}.

\begin{proof}[Proof of theorem~\ref{thm:globalpartialdataresult}]
\NTR{Replaced ``This proves the other direction" with ``This shows that $\xrt_1f=\xrt_1(\der\phi)=0$, and especially~$\xrt_1 f$ vanishes on all lines intersecting~$V$".}Let $f=\der\phi$ where $\phi\in\cdistr (\R^\dimens)$. Then $\der\phi\in(\cdistr(\R^\dimens))^\dimens$ and using the definition of the X-ray transform on distributions we obtain
\begin{equation}
\label{eq:xraytransformofpotential}
\langle \xrt_1(\der\phi), \varphi\rangle=\langle\der\phi, \xrt^*_1\varphi\rangle=\langle\phi, \diver(\xrt^*_1\varphi)\rangle=0.
\end{equation}
Here we used the fact that $\diver(\xrt^*_1\varphi)=0$ which follows from a straightforward computation. This shows that $\xrt_1f=\xrt_1(\der\phi)=0$, and especially~$\xrt_1 f$ vanishes on all lines intersecting~$V$. Assume then that $\der f|_V=0$. Since $\xrt_1 f=0$ on all lines intersecting~$V$ we obtain $\no_1 f|_V=0$. Theorem~\ref{thm:uniquecontinuationofnormaloperator} implies that $f=\der\phi$ for some $\phi\in\cdistr(\R^\dimens)$. This concludes the proof.
\end{proof}

\begin{proof}[Proof of theorem~\ref{thm:localpartialdataresult}]
\NTR{Explained the proof of the other direction of the claim more clearly.}If $f=\der\phi$ where $\phi\in H_0^1(\Omega)$, then using the same argument as in the proof of theorem~\ref{thm:globalpartialdataresult} and the fact that $H^1_0(\Omega)\subset\cdistr(\R^\dimens)$ in the sense of zero extension we obtain that $\xrt_1f=0$, and especially~$\xrt_1 f$ vanishes on all lines intersecting~$V$.
Then assume that $\der f|_V=0$ and $\xrt_1 f=0$ on all lines intersecting~$V$. Let $\widetilde{f}\in(\cdistr(\R^\dimens))^\dimens$ be the zero extension of~$f$.
The assumptions imply that $\der \widetilde{f}|_V=0$ and $\xrt_1\widetilde{f}=0$ on all lines intersecting
~$V$. Theorem~\ref{thm:globalpartialdataresult} implies that $\widetilde{f}=\der\phi$ for some $\phi\in\cdistr (\R^\dimens)$. Since $\Delta\phi=\diver (\widetilde{f})\in H^{-1}(\R^\dimens)$ we have $\phi\in H^1(\R^\dimens)$ by elliptic regularity.
On the other hand, $\spt (\phi)\subset \overline{\Omega}$ and hence $\phi\in H^1_0(\Omega)$ \cite[Theorem 3.33]{ML-strongly-elliptic-systems}.
The claim follows from the fact that $\der\phi=\widetilde{f}=f$ in~$\Omega$. 
\end{proof}

\begin{proof}[Proof of corollary~\ref{cor:ucpofgeneralnormaloperator}]
We know that the normal operator is $\no_A=A^T\circ\no_1\circ A$.
The assumptions imply that $\der (Af)|_V=\no_1(Af)|_V=0$. By theorem~\ref{thm:uniquecontinuationofnormaloperator} we obtain that $Af=\der\psi$ for some $\psi\in\cdistr (\R^\dimens)$.
This gives the claim.
\end{proof}

\begin{proof}[Proof of corollary~\ref{cor:partialdatageneraltransform}]
The claim follows directly from corollary~\ref{cor:ucpofgeneralnormaloperator} and from the fact that $\xrt_A=\xrt_1\circ A$.
\end{proof}

In theorems~\ref{thm:uniquecontinuationofnormaloperator} and~\ref{thm:globalpartialdataresult} one has $\spt(\phi)\subset\conv(\spt(f))$ where~$\conv(\spt(f))$ is the convex hull of~$\spt(f)$.
This follows from the fact that~$\phi$ has compact support and~$\der\phi$ vanishes in the connected set~$\conv(\spt (f))^c$.
This was pointed out in remark~\ref{rmk:spt}.

In corollaries~\ref{cor:ucpofgeneralnormaloperator} and~\ref{cor:partialdatageneraltransform} one also has~$\spt(\psi)\subset\conv(\spt(f))$.
This holds since~$\der\psi$ vanishes in the connected set~$\conv(\spt(Af))^c$ and $\spt(Af)=\spt(f)$.

\begin{proof}[Proof of corollary~\ref{cor:transverseraytransform}]
Assume first that $\diver(f)=0$. Since~$f$ is a covector-valued distribution in~$\R^2$ we can identify $\der f=\partial_1 f_2-\partial_2 f_1$. It follows that $\der (Bf)=-\diver (f)=0$ and thus $Bf=\der\eta$ for some $\eta\in\cdistr (\R^\dimens)$ by lemma~\ref{lemma:poincarelemma}. Therefore $\xrt_\perp f=\xrt_1(Bf)=\xrt_1(\der\eta)=0$. Assume then that $\diver (f)|_V=0$ and $\xrt_\perp f=0$ on all lines intersecting~$V$. As above we obtain that $\der (Bf)|_V=0$ and $\xrt_\perp f=0$ on all lines intersecting
~$V$. Corollary
~\ref{cor:partialdatageneraltransform} implies that $f=B^{-1}(\der\psi)$ for some $\psi\in\cdistr (\R^\dimens)$. From this we obtain that $\diver (f)=0$.

Assume then that $\der f|_V=\diver(f)|_V=0$ and both~$\xrt_1 f$ and~$\xrt_\perp f$ vanish on all lines intersecting~$V$. By the discussion above we obtain that $\diver(f)=0$. On the other hand, theorem~\ref{thm:globalpartialdataresult} implies that $f=\der\phi$ for some $\phi\in\cdistr(\R^2)$. Therefore $\Delta\phi=0$ and since~$\phi$ has compact support we must have $\phi=0$, i.e. $f=0$.
\end{proof}

\subsection{Proofs based on Stokes' theorem}
\label{sec:alternativeproof}

\NTR{Added reference to the article~\cite{AJM-v-line-transform}.}In this section we give alternative proofs for the partial data results using Stokes' theorem in~$\R^\dimens$. A similar approach was used in~\cite{JUH-principles-of-doppler-tomography, SSLP-doppler-tomography-vector-fields} in the case of full data, and also recently in~\cite{AJM-v-line-transform} for the generalized V-line transform. We prove the results first for compactly supported smooth one-forms and then use standard mollification argument to prove them for compactly supported covector-valued distributions. We only need to prove theorem~\ref{thm:globalpartialdataresult} since the rest of the partial data results follow from it. We will use the following lemma.

\begin{lemma}[{\cite[Theorem 1.2]{IM-unique-continuation-riesz-potential}}]
\label{lemma:partialdataproblemscalar}
Let $V\subset\R^\dimens$ be some nonempty open set and $g\in\cdistr (\R^\dimens)$. If $g|_V=0$ and $\xrt_0 g=0$ on all lines intersecting~$V$, then $g=0$.
\end{lemma}

\begin{proof}[Alternative proof of theorem~\ref{thm:globalpartialdataresult}]
By lemma~\ref{lemma:poincarelemma} it suffices to show that $\der f=0$.
Assume first that $n=2$ and $f\in (\csmooth(\R^2))^2$.
Let~$\gamma$ be any (oriented) line going through~$V$ and~$\nu$ the counterclockwise rotated normal to~$\gamma$.
We denote by $\gamma_h=h\nu+\cev{\gamma}$ the reversed parallel line shifted by $h>0$ in the direction of~$\nu$ so that~$\gamma_h$ also intersects~$V$.
By assumption $\int_\gamma f=\int_{\gamma_h}f=0$.

We form a closed loop~$\widetilde{\gamma}_h$ enclosing counterclockwise a rectangular region~$R_h$ such that the ends are outside~$\spt(f)$ (see figure~\ref{fig:green}).
When considered as chains, we have $\widetilde{\gamma}_h=\partial R_h$.
As the chains $\gamma-\gamma_h$ and~$\widetilde{\gamma}_h$ differ only outside the support of~$f$, the integrals coincide.
By Stokes' theorem
\begin{equation}
0
=
\int_\gamma f
-
\int_{\gamma_h}f
=
\int_{\widetilde{\gamma}_h}f
=
\int_{\partial R_h}f
=
\int_{R_h}\der f
=
\int_{R_h}\star\der f\,\der\mu,
\end{equation}
where~$\star$ is the Hodge star and~$\mu$ is the $2$-Hausdorff measure.

We aim to show that the scalar function~$\star\der f$ vanishes.
Scaling with~$h$, we find
\begin{equation}
0
=
\lim_{h\to0}
\frac1h
\int_{R_h}\star\der f\,\der\mu
=
\int_{\gamma}\star\der f\,\der s.
\end{equation}
Now that $\star\der f|_V=0$ and $\xrt_0(\star\der f)(\gamma)=0$ for all lines~$\gamma$ meeting~$V$, lemma~\ref{lemma:partialdataproblemscalar} implies that $\star\der f=0$ and thus also $\der f=0$ in the whole plane.

Consider then the case $\dimens\geq3$ for a compactly supported smooth one-form~$f$.
Let $P\subset\R^\dimens$ be any two-plane meeting~$V$ and $\iota_P\colon P\to\R^\dimens$ the corresponding inclusion.
By the argument above for the two-form~$\iota^*_Pf$ in the plane~$P$ we have that $\iota_P^*\der f=0$ for all such planes.

Take any point $z\in\R^\dimens$.
For any plane~$P$ through~$z$ that intersects~$V$ we have $\iota_P^*\der f=0$.
This is an open subset of the Grassmannian of $2$-planes through~$z$, so $\der f(z)=0$.
As the point~$z$ was arbitrary, we have $\der f=0$.

Finally, let $f\in(\cdistr(\R^\dimens))^
\dimens$ and define $f_{\epsilon}=f\ast j_{\epsilon}=(f_1\ast j_{\epsilon}, \dotso , f_n\ast j_{\epsilon})$ where $j_{\epsilon}\in \csmooth(\R^\dimens)$ is the standard mollifier.
Then $f_{\epsilon}\in (\csmooth(\R^\dimens))^\dimens$ and $\langle \xrt_1 (f\ast j_{\epsilon}), \varphi\rangle=\langle\xrt_1 f, \xrt_0 j_{\epsilon}\circledast\varphi\rangle$ where
\begin{equation}
(h\circledast g)(z, \theta)=\int_{\theta^{\perp}}h(z-y, \theta)g(y, \theta)\der y.
\end{equation}
Hence there is a nonempty open set $W\subset V$ such that for small $\epsilon >0$ we have $f_{\epsilon}|_W=0$ and $\xrt_1 f_{\epsilon}=0$ on all lines intersecting $W$.
Using the above reasoning for smooth one-forms we obtain $0=\der f_{\epsilon}=\der f\ast j_{\epsilon}$ for small $\epsilon>0$. Taking $\epsilon\rightarrow 0$ we get $\der f=0$.
\end{proof}

\begin{figure}[htp]
\centering
\includegraphics[height=7.3cm]{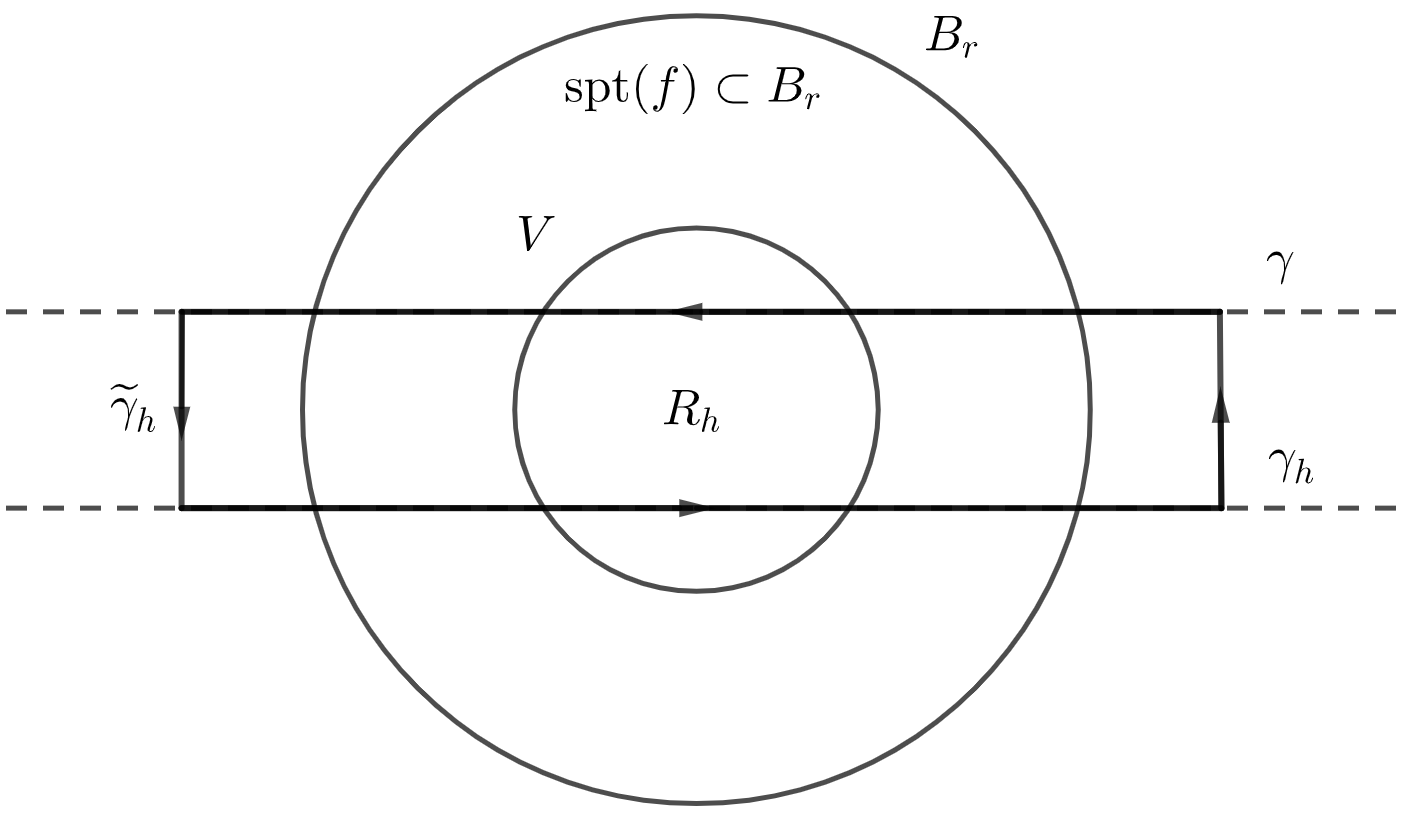}
\caption{Basic idea of the alternative proof of theorem~\ref{thm:localpartialdataresult} when $\dimens=2$. We may assume that~$f$ is supported in a ball~$B_r$. We form a closed loop~$\widetilde{\gamma}_h$ with the lines~$\gamma$ and~$\gamma_h$ (dashed) enclosing the rectangular region~$R_h$. Then we apply Stokes' theorem and a limit argument $h\rightarrow 0$ together with a known partial data result for scalar fields to obtain that $\der f=0$.}
\label{fig:green}
\end{figure}

Now the proof of theorem~\ref{thm:localpartialdataresult} follows in the same way from theorem~\ref{thm:globalpartialdataresult} as before using the zero extension~$\widetilde{f}$.
Corollaries~\ref{cor:partialdatageneraltransform} and~\ref{cor:transverseraytransform} are also direct consequences of theorem~\ref{thm:globalpartialdataresult} since $\xrt_A=\xrt_1\circ A$.

\NTR{We explained the proof of the support theorem in more detail.}
Moreover, the above alternative proof can be used to prove a complementary support theorem for the transform~$\xrt_A$: if $\der (Af)|_C=0$ and $\xrt_Af=\xrt_1(Af)=0$ on all lines not intersecting a convex and compact set~$C$, then $f=A^{-1}(\der\psi)$ for some potential~$\psi$ (see~\cite[Theorem 7.5]{SU:microlocal-analysis-integral-geometry} for a similar support theorem for the X-ray transform~$\xrt_1$). Indeed, if~$\gamma$ is any line not intersecting~$C$, then we can form a closed loop~$\widetilde{\gamma}_h$ as in figure~\ref{fig:green} so that the loop is completely contained in~$C^c$ and the ends are outside the support of~$f$. Using Stokes' theorem and a limit argument $h\rightarrow 0$ as in the alternative proof above we obtain that $\xrt_0(\star\der (Af))=0$ on all lines not intersecting~$C$. Now we can use the Helgason support theorem for scalar fields (see e.g.~\cite[Corollary 6.1]{HE:integral-geometry-radon-transforms} and~\cite[Section 5.2]{SU:microlocal-analysis-integral-geometry}) to conclude that $\der (Af)=0$ in~$C^c$. Since also $\der(Af)|_C=0$ we get that~$Af$ is a closed one-form and thus exact, i.e. there is a scalar field~$\psi$ such that $Af=\der\psi$.

\NTR{References have been updated.}

\bibliography{refs} 
\bibliographystyle{abbrv}

\end{document}